\documentclass[12pt,british,reqno]{amsart}

\usepackage{babel}

\usepackage[T1]{fontenc}
\DeclareFontFamily{T1}{calligra}{}
\DeclareFontShape{T1}{calligra}{m}{n}{<->s*[1.44]callig15}{}
\DeclareMathAlphabet\mathrsfso      {U}{rsfso}{m}{n}
\usepackage{amsthm,amssymb,amsmath} 
\usepackage[latin1]{inputenc}
\usepackage{enumerate}
\usepackage[all]{xy}

\usepackage{mathtools}

\makeatletter
\def\@map#1#2[#3]{\mbox{$#1 \colon\thinspace #2 \longrightarrow #3$}}
\def\map#1#2{\@ifnextchar [{\@map{#1}{#2}}{\@map{#1}{#2}[#2]}}
\g@addto@macro\@floatboxreset\centering
\makeatother

\renewcommand{\epsilon}{\ensuremath{\varepsilon}}
\renewcommand{\phi}{\ensuremath{\varphi}}

\renewcommand{\to}{\ensuremath{\longrightarrow}}
\renewcommand{\mapsto}{\ensuremath{\longmapsto}}

\newcommand{\Z}{\ensuremath{\mathbb{Z}}}

\newcommand{\FF}{\ensuremath{\mathbb F}}
\newcommand{\F}[1][n]{\ensuremath{\FF_{{#1}}}}

\newcommand{\hyparr}[1]{\ensuremath{\mathrsfso{#1}}}

\newcommand{\brak}[1]{\ensuremath{\left\{ #1 \right\}}}
\newcommand{\ang}[1]{\ensuremath{\left\langle #1\right\rangle}}

\newtheoremstyle{theoremm}{}{}{\itshape}{}{\scshape}{.}{ }{}
\theoremstyle{theoremm}
\newtheorem{thm}{Theorem}
\newtheorem{lem}[thm]{Lemma}
\newtheorem{prop}[thm]{Proposition}
\newtheorem{cor}[thm]{Corollary}
\newtheoremstyle{remark}{}{}{}{}{\scshape}{.}{ }{}

\theoremstyle{remark}
\newtheorem{defn}[thm]{Definition}
\newtheorem{rem}[thm]{Remark}
\newtheorem{rems}[thm]{Remarks}

\newcommand{\aut}[1]{\ensuremath{\operatorname{\text{Aut}}\left({#1}\right)}}

\newcommand{\redefn}[1]{Definition~\protect\ref{defn:#1}}
\newcommand{\rethm}[1]{Theorem~\protect\ref{thm:#1}}
\newcommand{\relem}[1]{Lemma~\protect\ref{lem:#1}}
\newcommand{\reprop}[1]{Proposition~\protect\ref{prop:#1}}

\newcommand{\resec}[1]{Section~\protect\ref{sec:#1}}

\numberwithin{equation}{section}
\allowdisplaybreaks

\usepackage[pdftex,%
bookmarks=true,
breaklinks=true,
bookmarksnumbered = true,%     % Signets numÃÂÃÂ©rotÃÂÃÂ©s
colorlinks= true,%     % Liens en couleur
urlcolor= green,
anchorcolor = yellow,
citecolor=blue,%  % Couleur des liens externes
]{hyperref}                   % Utilisation de HyperTeX

\usepackage[x11names]{xcolor}
\begin{document}

\title[Positive cones and bi-orderings on almost-direct products of free groups]{Positive cones and bi-orderings on almost-direct products of free groups}

\author[O.~Ocampo]{Oscar Ocampo}
\address{Universidade Federal da Bahia, Departamento de Matem\'atica - IME, Av.~Milton Santos~S/N, CEP:~40170-110 - Salvador - BA - Brazil}
\email{oscaro@ufba.br}

\author[J.~R.~Theodoro~de~Lima]{Juliana Roberta Theodoro de Lima}
\address{Instituto de Matem\'atica, Universidade Federal de Alagoas, Avenida Lourival Melo Mota, s/n, Tabuleiro dos Martins, 57072-900, Macei\'o-AL, Brazil}
\email{juliana.lima@im.ufal.br}

\date{\today}


\date{\today}

\begin{abstract}
\noindent 
Almost-direct products of free groups arise naturally in braid theory and in the study of automorphism groups of free groups. Although bi-invariant orderings are known to exist for many such groups, their explicit structure is often left implicit.
In this paper, we give an explicit description of the positive cones defining bi-invariant orderings on almost-direct products of free groups, using normal forms derived from the almost-direct product decomposition together with Magnus-type orderings on free factors. We establish key structural properties of these cones, including compatibility with natural projections, convexity of canonical subgroups, and invariance under suitable classes of automorphisms. 
As applications, we show how the construction applies to several families of groups of geometric and algebraic interest, such as pure monomial braid groups and McCool groups.

\end{abstract}
\maketitle

\section{Introduction}

Orderable groups play an important role in several areas of group theory and topology. A left-invariant ordering on a group has strong algebraic consequences, such as the absence of zero divisors in the group ring over any field, while the existence of a bi-invariant ordering implies that the group embeds into a division ring.
Beyond these algebraic aspects, orderable groups arise naturally in low-dimensional topology, braid theory, and the study of group actions on one-dimensional manifolds; see, for instance, \cite{CR,DNR,RW}.

Among the classes of groups that frequently appear in these contexts are \emph{almost-direct products of free groups}, namely iterated semidirect products
\[
G \;=\; F_{n_k} \rtimes F_{n_{k-1}} \rtimes \cdots \rtimes F_{n_1},
\]
where the action of each factor on the abelianization of the preceding one is trivial. Such groups occur naturally as fundamental groups of complements of hyperplane arrangements, orbit configuration spaces, and in various subgroups of automorphism groups of free groups; see, for example, \cite{C1,CPVW,FR1,FR2,P,X}.

It is well known that many almost-direct products of free groups admit bi-invariant orderings. This fact can be established using residual properties, or by applying general results on iterated semidirect products with IA-actions, as developed by Kim and Rolfsen \cite{KR} and further extended in related settings by Yurasovskaya \cite{Y}. 
From this point of view, the existence of bi-orderings on these families is by now well understood, thanks in large part to the work of Kim and Rolfsen \cite{KR}.

The purpose of the present article is to study the explicit structure of bi-invariant orderings on almost-direct products of free groups. More precisely, we focus on the description of the \emph{positive cones} defining such orderings, viewed as concrete algebraic objects that encode the order in an explicit and structurally meaningful way. 
While positive cones are implicit in many constructions of bi-orderings, they are rarely made explicit or systematically exploited.

Our contribution is a systematic description of positive cones for bi-invariant orderings on almost-direct products of free groups. These cones are defined using normal forms associated with the almost-direct product decomposition, together with Magnus-type orderings on the free factors \cite{MKS}. This approach yields bi-orderings that are explicit: the sign of an element is determined by the highest nontrivial component in its normal form.

In addition, we investigate several structural properties of the resulting positive cones. In particular, we show that they are compatible with natural projection maps arising from the almost-direct product structure, that canonical subgroups associated with the decomposition are convex, and that the cones are invariant under suitable classes of automorphisms; see also \cite{BN,BNS,RW} for related phenomena in other contexts. These properties highlight the robustness of the constructed orderings and make them suitable for further applications.

As concrete illustrations of our general results, we show how the construction applies to several families of groups of geometric and algebraic interest. These include pure monomial braid groups and McCool groups, which naturally arise in the study of braid groups, configuration spaces, and automorphism groups of free groups; see, for example, \cite{Mc}.

This paper is organized as follows. In Section~\ref{sec:adfreegroups} we recall the definition and basic properties of almost-direct products of free groups, including the existence of normal forms. Section~\ref{sec:cones} contains the construction of the positive cone and the proof that it defines a bi-invariant ordering (Theorem~\ref{thm:main_cone}). In Section~\ref{sec:properties} we establish structural properties of the cone: compatibility with projections, convexity of canonical subgroups, and invariance under automorphisms. Section~\ref{sec:examples} illustrates the construction with explicit families of groups, including pure monomial braid groups, McCool groups, and fundamental groups of hypersolvable arrangements. Finally, Section~\ref{sec:remarks} collects some concluding remarks and perspectives. An appendix extends the construction to almost-direct products of reduced free groups.

The results presented here provide a unified and effective framework for the study of bi-invariant orderings on almost-direct products of free groups. By providing an explicit description of the associated positive cones, our approach opens the way to further investigations in group theory, topology, and dynamics, where such orderings play a fundamental role.

\subsection*{Acknowledgments}

The first author gratefully acknowledges the support of Eliane Santos, the staff of HCA, Bruno Noronha, Luciano Macedo, M\'arcio Isabella, Andreia de Oliveira Rocha, Andreia Gracielle Santana, Ednice de Souza Santos, and SMURB--UFBA (Servi\c{c}o M\'edico Universit\'ario Rubens Brasil Soares), whose assistance since July~2024 was essential in enabling the completion of this work. 
O.~O.~was partially supported by the National Council for Scientific and Technological Development (CNPq, Brazil) through a \textit{Bolsa de Produtividade} grant No.~305422/2022--7.

\section{Almost-direct products of free groups}\label{sec:adfreegroups}

In this section we recall the notion of almost-direct products of free groups and fix the notation that will be used throughout the paper.  We also review basic structural properties of these groups, including the existence of canonical normal forms, which will play a central role in the construction of explicit bi-invariant orderings.

\subsection{Definition and basic properties}

Let $G$ be a group admitting a decomposition as an iterated semidirect product
\[
G \;=\; F_{n_k} \rtimes F_{n_{k-1}} \rtimes \cdots \rtimes F_{n_1},
\]
where each $F_{n_i}$ is a free group of finite rank $n_i$. We say that $G$ is an \emph{almost-direct product of free groups} if, for each $i \geq 2$, the action of
\[
F_{n_{i-1}} \rtimes \cdots \rtimes F_{n_1}
\]
on $F_{n_i}$ induces the trivial action on the abelianization
\[
H_1(F_{n_i},\mathbb{Z}) \;=\; F_{n_i}/[F_{n_i},F_{n_i}].
\]
Equivalently, the corresponding homomorphism
\[
F_{n_{i-1}} \rtimes \cdots \rtimes F_{n_1} \longrightarrow \aut{F_{n_i}}
\]
has image contained in the group of IA-automorphisms of $F_{n_i}$. Here, an automorphism of $F_{n_i}$ is said to be an \emph{IA-automorphism} if it induces the identity on $H_1(F_{n_i},\Z)$.

This notion was introduced and systematically studied in the context of configuration spaces and hyperplane arrangements; see, for instance, \cite{C1,C2,CS,FR1,FR2,P}. Almost-direct products of free groups enjoy several favorable algebraic and homological properties, including residual nilpotence and torsion-freeness in many cases.

\subsection{Normal forms}

A fundamental feature of almost-direct products of free groups is the existence of canonical normal forms. Every element $g \in G$ can be written uniquely as
\[
g \;=\; g_k\, g_{k-1}\cdots g_1,
\qquad g_i \in F_{n_i}.
\]
This expression will be referred to as the \emph{normal form} of $g$ with respect to the chosen almost-direct product decomposition.

The triviality of the induced actions on abelianizations implies that conjugation by elements of $F_{n_{i-1}} \rtimes \cdots \rtimes F_{n_1}$ preserves the lower central series of $F_{n_i}$. In particular, commutators in the factors $F_{n_i}$ behave well with respect to the above normal form, a feature that will play a key role in the construction of explicit bi-invariant orderings.

\subsection{Examples}

We briefly recall some standard examples of almost-direct products of free groups that will be relevant later in the paper.

\begin{itemize}
\item[(i)] \emph{Pure braid groups.}
It is classical that the pure braid group $P_n$ admits a decomposition as an iterated semidirect product of free groups, known as the Artin combing; see \cite{FR2,KR}. Moreover, the induced actions are trivial on abelianizations, so $P_n$ is an almost-direct product of free groups.

\item[(ii)] \emph{Pure monomial braid groups.}
Pure monomial braid groups arise as fundamental groups of orbit configuration spaces associated with finite group actions. They admit decompositions as almost-direct products of free groups, as shown using techniques from configuration spaces and hyperplane arrangements; see \cite{C1,X}.
These groups will provide one of the main classes of examples in the present work.

\item[(iii)] \emph{McCool groups.}
The McCool group $Cb_n$, consisting of basis-conjugating automorphisms of a free group, and its upper triangular subgroup $Cb_n^+$, admit natural decompositions as almost-direct products of free groups; see \cite{CPVW,Mc}.
These groups play an important role in the study of automorphism groups of free groups and will also be considered in later sections.
\end{itemize}

Throughout the paper, unless stated otherwise, we fix an almost-direct product decomposition of $G$ and work with the associated normal forms.

\section{Positive cones and bi-invariant orderings}\label{sec:cones}

This section is devoted to the construction of explicit positive cones defining bi-invariant orderings on almost-direct products of free groups. After recalling basic facts on positive cones and Magnus-type orderings on free groups, we introduce a lexicographic construction adapted to the almost-direct product structure and prove that it yields a bi-invariant ordering.

\subsection{Positive cones and bi-invariant orderings}

We begin by recalling the notion of a positive cone associated with a bi-invariant ordering.

\begin{defn}\label{defn:positive_cone}
Let $G$ be a group.
A subset $P \subset G$ is called a \emph{positive cone} if the following conditions hold:
\begin{itemize}
\item[(i)] $P \cdot P \subset P$;
\item[(ii)] $G = P \sqcup \brak{1} \sqcup P^{-1}$;
\item[(iii)] $P$ is invariant under conjugation, that is,
$gPg^{-1} = P$ for all $g \in G$.
\end{itemize}
In this case, the relation defined by
\[
g < h \quad \Longleftrightarrow \quad g^{-1}h \in P
\]
is a bi-invariant ordering on $G$.
\end{defn}

Conversely, every bi-invariant ordering on a group $G$ determines a positive cone satisfying the above properties. Thus, the study of bi-invariant orderings on $G$ is equivalent to the study of its positive cones; see, for instance, \cite{B,CR}.

\subsection{Magnus-type orderings on free groups}

Let $F_n$ be the free group of rank $n$. We briefly recall that Magnus expansions give rise to bi-invariant orderings on free groups; see \cite{KR,MKS}. Fix a free generating set for $F_n$ and  let $\Z\langle\!\langle X_1,\dots,X_n\rangle\!\rangle$ denote the ring of formal power series in noncommuting variables.
Consider the Magnus embedding
\[
F_n \longrightarrow 1 + \langle X_1,\dots,X_n\rangle \subset \Z\langle\!\langle X_1,\dots,X_n\rangle\!\rangle.
\]
Ordering the target ring lexicographically yields a bi-invariant ordering on $F_n$, whose associated positive cone will be denoted by ${\mathcal P}(F_{n}) \subset F_n$. This ordering will be referred to as a \emph{Magnus-type ordering}.

Throughout the paper, whenever a free factor $F_{n_i}$ appears in an almost-direct product decomposition, it will be equipped with a fixed Magnus-type ordering and corresponding positive cone ${\mathcal P}(F_{n_i})$.

\subsection{Construction of the positive cone}

Let
\[
G \;=\; F_{n_k} \rtimes F_{n_{k-1}} \rtimes \cdots \rtimes F_{n_1}
\]
be an almost-direct product of free groups, as in \resec{adfreegroups}. Every element $g \in G$ admits a unique normal form
\[
g \;=\; g_k g_{k-1}\cdots g_1, \qquad g_i \in F_{n_i}.
\]

\begin{defn}\label{defn:cone_ADP}
Define a subset $P \subset G$ as follows. An element $g = g_k g_{k-1}\cdots g_1$ belongs to $P$ if and only if there exists an index $j \in \brak{1,\dots,k}$ such that
\[
g_j \in {\mathcal P}(F_{n_j})
\quad\text{and}\quad g_i = 1 \text{ for all } i > j.
\]
\end{defn}

In other words, $P$ consists of those elements whose highest nontrivial component in the normal form is positive with respect to the fixed Magnus-type ordering on the corresponding free factor. 

\begin{rem}
The positive cone $P$ is defined lexicographically with respect to the normal form in the almost-direct product decomposition. In particular, an element may have several nontrivial components belonging to the positive cones of lower-index free factors and still be positive in $G$, provided that its highest-index nontrivial component is positive, independently of the lower components. Conversely, the sign of an element is determined solely by this highest-index nontrivial component. 

The choice of the highest-index nontrivial component reflects the hierarchical structure of the iterated semidirect product.
\end{rem}

\subsection{Main result}

We now state and prove the main result of this section. It shows that the lexicographic construction introduced above yields a bi-invariant ordering on any almost-direct product of free groups, defined by an explicit and computable positive cone. This result provides a concrete realization of bi-orderability in this setting and serves as the foundation for the structural properties and applications developed in the subsequent sections.

\begin{thm}\label{thm:main_cone}
Let $G$ be an almost-direct product of free groups. The subset $P \subset G$ defined in \redefn{cone_ADP} is a positive cone.  In particular, it defines a bi-invariant ordering on $G$.
\end{thm}

\begin{proof}
We verify the three defining properties of a positive cone.

\emph{(i) Closure under multiplication.}
Let $g,h \in P$, and write
\[
g = g_k\cdots g_1, \qquad h = h_k\cdots h_1
\]
in normal form. Let $j$ (resp.\ $j'$) be the largest index such that $g_j \neq 1$ (resp.\ $h_{j'} \neq 1$). 
If $j > j'$, then the highest nontrivial component of $gh$ is $g_j \in {\mathcal P}(F_{n_j})$.
If $j = j'$, then the highest nontrivial component of $gh$ is $g_j h_j$, which belongs to ${\mathcal P}(F_{n_j})$ since ${\mathcal P}(F_{n_j})$ is a positive cone in $F_{n_j}$.
The case $j < j'$ is analogous.
Hence $gh \in P$.

\emph{(ii) Trichotomy.}
Every nontrivial element $g \in G$ has a highest nontrivial component $g_j$ in its normal form. Since $P_{n_j}$ is a positive cone in $F_{n_j}$, either $g_j \in P_{n_j}$ or $g_j^{-1} \in P_{n_j}$, implying that either $g \in P$ or $g^{-1} \in P$. Thus $G = P \sqcup \brak{1} \sqcup P^{-1}$.

\emph{(iii) Conjugation invariance.}
Let $x \in G$ and $g \in P$. Write $g = g_k\cdots g_1$ and let $j$ be the highest index such that $g_j \neq 1$. Since $G$ is an almost-direct product, conjugation by $x$ acts trivially on the abelianization of $F_{n_j}$. In particular, conjugation preserves the Magnus-type ordering on $F_{n_j}$. Hence the highest nontrivial component of $xgx^{-1}$ lies again in ${\mathcal P}(F_{n_j})$, and therefore $xgx^{-1} \in P$.

This shows that $P$ is a positive cone, completing the proof.
\end{proof}

The bi-invariant ordering determined by the cone $P$ will be referred to as the \emph{lexicographic bi-ordering} associated with the chosen almost-direct product decomposition and the fixed Magnus-type orderings on the free factors.

\begin{cor}\label{cor:group_ring}
Let $G$ be an almost-direct product of free groups. Then $G$ is torsion-free and has no generalized torsion. Moreover, the integral group ring $\Z G$ embeds into a skew field.
\end{cor}

\begin{proof}
By \rethm{main_cone}, the group $G$ admits a bi-invariant ordering. It is well known that bi-orderable groups are torsion-free and do not contain generalized torsion elements; see, for example, \cite{CR}. Furthermore, by a classical result of Malcev and subsequent refinements, the group ring of a bi-orderable group embeds into a division ring. In particular, $\Z G$ embeds into a skew field.
\end{proof}

\section{Structural properties of the positive cone}\label{sec:properties}

In this section we investigate structural properties of the positive cones constructed in the previous section. In particular, we study their compatibility with natural projection maps, the convexity of canonical subgroups arising from the almost-direct product decomposition, and their behavior under suitable classes of automorphisms.

Throughout this section, let
\[
G \;=\; F_{n_k} \rtimes F_{n_{k-1}} \rtimes \cdots \rtimes F_{n_1}
\]
be an almost-direct product of free groups equipped with the lexicographic bi-ordering defined by the positive cone $P \subset G$ introduced in \redefn{cone_ADP}. We write ${\mathcal P}(F_{n_i}) \subset F_{n_i}$ for the positive cone of the fixed Magnus-type ordering on each free factor.

\subsection{Compatibility with natural projections}

For each $1 \leq j \leq k$, consider the natural projection
\[
\pi_j \colon G \to F_{n_j}
\]
given by forgetting all components except the $j$-th one in the normal form.

\begin{prop}\label{prop:projection}
Let $g \in G$ be written in normal form as $g = g_k\cdots g_1$. If $g \in P$ and $j$ is the largest index such that $g_j \neq 1$, then
\[
\pi_j(g) = g_j \in {\mathcal P}(F_{n_j}).
\]
In particular, the ordering on $G$ restricts to the given Magnus-type ordering on each free factor.
\end{prop}

\begin{proof}
By definition of the cone $P$, the highest nontrivial component $g_j$ of $g$ lies in ${\mathcal P}(F_{n_j})$. Since $\pi_j(g) = g_j$, the claim follows immediately.
\end{proof}

This compatibility shows that the lexicographic bi-ordering on $G$ extends the chosen orderings on the free factors in a coherent way.

\subsection{Convex subgroups}

Convexity of natural subgroups is a fundamental feature of lexicographic orderings.

\begin{defn}
Let $(G,<)$ be an ordered group. A subgroup $H \leq G$ is said to be \emph{convex} if, whenever $h_1 < g < h_2$ for some $h_1,h_2 \in H$ and $g \in G$, then $g \in H$.
\end{defn}

\begin{prop}\label{prop:convex}
For each $1 \leq j \leq k$, the subgroup
\[
G_j \;=\; F_{n_j} \rtimes F_{n_{j-1}} \rtimes \cdots \rtimes F_{n_1}
\]
is convex in $G$ with respect to the lexicographic bi-ordering defined by $P$.
\end{prop}

\begin{proof}
Let $g \in G$ and write $g = g_k\cdots g_1$ in normal form. If $g \notin G_j$, then there exists an index $i > j$ such that $g_i \neq 1$. Hence the highest nontrivial component of $g$ lies in some $F_{n_i}$ with $i>j$. It follows that either $g > h$ for all $h \in G_j$ or $g < h$ for all $h \in G_j$, depending on whether $g_i \in {\mathcal P}(F_{n_i})$ or $g_i^{-1} \in {\mathcal P}(F_{n_i})$. In particular, no element of $G$ can lie strictly between two elements of $G_j$ in the order unless it already belongs to $G_j$.
Thus $G_j$ is convex.
\end{proof}

As a consequence, the lexicographic bi-ordering on $G$ restricts to each subgroup $G_j$ and is compatible with the natural projection $G_j \to F_{n_j}$, inducing the fixed Magnus-type ordering on the free factor $F_{n_j}$.

\subsection{Stability under automorphisms}

We conclude this section by observing that the constructed positive cone behaves well under a natural class of automorphisms.

Let $\phi \in \aut{G}$ be an automorphism preserving the almost-direct product decomposition, that is, $\phi(G_j)=G_j$ for all $j$, with $G_j \;=\; F_{n_j} \rtimes F_{n_{j-1}} \rtimes \cdots \rtimes F_{n_1}$.

\begin{prop}\label{prop:aut}
Suppose that $\phi \in \aut{G}$ satisfies the following properties:
\begin{itemize}
\item[(i)] $\phi(F_{n_j}) = F_{n_j}$ for all $j$;
\item[(ii)] the restriction $\phi|_{F_{n_j}}$ preserves the Magnus-type ordering on $F_{n_j}$.
\end{itemize}
Then $\phi(P)=P$. In particular, $\phi$ preserves the lexicographic bi-ordering on $G$.
\end{prop}

\begin{proof}
Let $g \in P$ and write $g = g_k\cdots g_1$ in normal form. Let $j$ be the largest index such that $g_j \neq 1$.
By assumption, $\phi(g_i)=1$ for all $i>j$, and $\phi(g_j) \in {\mathcal P}(F_{n_j})$ since $\phi|_{F_{n_j}}$ preserves the Magnus-type ordering. Hence $\phi(g)$ has highest nontrivial component in ${\mathcal P}(F_{n_j})$, which implies that $\phi(g)\in P$.
\end{proof}

\begin{rem}
Automorphisms satisfying the hypotheses of \reprop{aut} naturally arise in the study of automorphism groups of pure braid groups and McCool groups; see, for example, \cite{BN,BNS,CPVW}.
\end{rem}

\section{Examples}\label{sec:examples}

In this section we illustrate the general construction developed in \resec{cones} and \resec{properties} by describing explicit positive cones and bi-invariant orderings for several families of groups of geometric and algebraic interest.
In each case, the group under consideration admits a decomposition as an almost-direct product of free groups, and the corresponding ordering is obtained by applying the lexicographic construction introduced earlier.

\subsection{Pure monomial braid groups}

Pure monomial braid groups arise naturally as fundamental groups of orbit configuration spaces associated with finite group actions.
They can be viewed as natural generalizations of pure braid groups and have been studied from both algebraic and topological perspectives; see, for instance, \cite{C1,X}.

Let $M$ be a manifold and let $G$ be a discrete group that acts properly discontinuosly on $M$. Let 
$$
F_G(M,n)=\{ (x_1,\ldots, x_n)\in M^n \mid G\cdot x_i \cap G\cdot x_j = \emptyset \textrm{ if } i\neq j \}
$$
denote the \emph{orbit configuration space of $M$ under the action of $G$}, which consits of all ordered $n$-tuples of points in $M$ which lies in distinct orbits. These spaces were introduced in \cite{X} and have special features in their loop space homology and features of certain Lie algebras.

Let $Q_n^G$ denote the union of $n$ distinct orbits, $G\cdot x_1, \ldots, G\cdot x_n$, in $M$. From \cite[Theorem~2.2]{X}, there are fibrations $F_G(M,n) \to F_G(M,i)$ with fiber over the point $(p_1, p_2,\ldots, p_i)$ in $F_G(M,i)$ given by $F_G(M\setminus Q_n^G,n-i)$. This result is a natural generalization to orbit configuration spaces of the well known Fadell-Neuwirth theorem for configuration spaces \cite[Theorem~3]{FN}.

The orbit configuration space $F_G(\mathbb{C}^{\ast},n)$, where $G=\Z/r\Z$, was discussed in \cite{C1}. It is the complement of the reflection arrangement associated to the (full) monomial group $G(r,n)$, the complex reflection group isomorphic to the wreath product of the symmetric group $S_n$ and $G=\Z/r\Z$. Its fundamental group, denoted by $P(r,n)=\pi_1(F_G(\mathbb{C}^{\ast},n))$, is called the \emph{pure monomial braid group}. 
The special case $r=2$ is a fiber type hyperplane arrangement, its fundamental group considered separately in \cite[Theorem~1.4.3]{C1} was called the \emph{Brieskorn generalized pure braid group} (see also \cite[Remark~2.2.5]{C1}).

It follows from \cite[Theorem~2.1.3 (item 4) and Proposition~2.2.2]{C1} that the pure monomial braid groups $P(r,n)=\rtimes_{j=1}^n F_{r(j-1)+1}$ are almost-direct products of free groups. 
Consequently, $P(r,n)$ falls within the class of groups considered in the previous
sections.

Fix Magnus-type orderings on each free factor $F_{n_j}$ and denote by ${\mathcal P}(F_{n_j})$ the corresponding positive cones.
The positive cone $P \subset P(r,n)$ is then defined as in \redefn{cone_ADP}: an element of $P(r,n)$ is positive if and only if the highest nontrivial component in its normal form lies in ${\mathcal P}(F_{n_j})$ for the corresponding index $j$.

\begin{prop}\label{prop:Prn}
The subset $P \subset P(r,n)$ defined above is a positive cone. In particular, it determines a bi-invariant ordering on the pure monomial braid group
$P(r,n)$.
\end{prop}

\begin{proof}
This is a direct application of \rethm{main_cone}, since $P(r,n)$ is an almost-direct product of free groups and the chosen orderings on the free factors are Magnus-type.
\end{proof}

\begin{rem}
The description of the positive cone given above is explicit: given an element of $P(r,n)$ written in normal form, its sign is determined solely by the sign of its highest nontrivial free component.
\end{rem}

\subsection{Subgroups of the automorphism group of free groups.}

The automorphism group of a free group is one of the most interesting groups in combinatorial group theory. This group and many of their subgroups were deeply studied however many questions still remain open. In this section we consider some of its subgroups. 

Let $F_n$ be a free group of rank $n\geq 2$ generated by $n$ letters $\{ x_1, x_2, \ldots, x_n \}$ and $\aut{F_n}$ be its automorphism group. For any two elements $a$ and $b$ of a group $G$ its commutator is given by $[a,b]=a^{-1}b^{-1}ab$ and let $[G, G]$ denote the commutator subgroup of $G$. The group $\aut{F_n/[F_n, F_n]}$ is isomorphic to the general linear group $GL_n(\Z)$ over the ring of integers.  The kernel of the natural map 
$$
\aut{F_n} \to GL_n(\Z)
$$
consists of automorphisms acting identically modulo the commutator subgroup $[F_n, F_n]$ is called the \emph{IA-automorphism group} and is denoted by $IA_n$, see \cite{MKS}.
Therefore, the following short exact sequence holds
$$
1\to IA_n \to \aut{F_n} \to GL_n(\Z) \to 1.
$$ 

Nielsen (for $n\leq 3$) and Magnus (for all $n$) showed that the group $IA_n$ is generated by automorphisms
$$
\epsilon_{ijk}\colon
\begin{cases}
x_i \mapsto x_i[x_j, x_k] & \textrm{with } i\neq j,k,\, j>i,\\
x_l\mapsto x_l & \textrm{with } l\neq i,
\end{cases}
\quad 
\textrm{and}
\quad
\epsilon_{ij}\colon
\begin{cases}
x_i \mapsto x_j^{-1}x_ix_j & \textrm{with } i\neq j,\\
x_l\mapsto x_l & \textrm{with } l\neq i,
\end{cases}
$$
see \cite{MKS}.

\subsubsection{McCool groups}

Let us consider the subgroup of $IA_n$ generated by the automorphisms $\epsilon_{ij}$, for $1\leq i\neq j\leq n$. It is called the \emph{basis-conjugating automorphism group} and is denoted by $Cb_n$. As mentioned in the introduction of \cite{CPVW} this subgroup has topological interpretations as \emph{the pure of motions of $n$ unlinked circles in $S^3$} and so it is known as the ``group of loops'' or ``loop braid group'',  and it is also known under the name of \emph{the pure braid-permutation group}. A presentation for $Cb_n$ was obtained by McCool \cite{Mc}, and it is also listed in \cite[Theorem~1.1]{CPVW} and \cite[equations (1)-(3)]{BN}.

The subgroup of $Cb_n$ generated by the $\epsilon_{ij}$, for $1\leq i<j\leq n$, denoted by $Cb_n^+$ here, is called the \emph{upper triangular McCool groups} or just \emph{upper McCool groups}.   By \cite[Theorem~1.2]{CPVW} we have the decomposition of $Cb_n^+$ described as an almost-direct product of free groups $Cb_n^+=\rtimes_{j=2}^n F_{j-1}$.  Thus $Cb_n^+$ fits naturally into the framework developed in this paper.

Fix Magnus-type orderings on each free factor $F_{m_j}$ and denote by ${\mathcal P}(F_{m_j})$ the associated positive cones. The positive cone $P \subset Cb_n^+$ is defined lexicographically as in \redefn{cone_ADP}.

\begin{prop}\label{prop:McCool}
The subset $P \subset Cb_n^+$ defines a bi-invariant ordering on the McCool group $Cb_n^+$.
\end{prop}

\begin{proof}
The proof follows immediately from \rethm{main_cone}, since $Cb_n^+$ is an almost-direct product of free groups with IA-actions.
\end{proof}

\begin{rem}
The resulting bi-ordering on $Cb_n^+$ is compatible with the natural filtration by subgroups arising from the almost-direct product decomposition, and the corresponding subgroups are convex by \reprop{convex}.
\end{rem}

\subsubsection{The partial inner automorphism group $I_n$}\label{subsec:In}

Let $F_n=\F[n]$ be the free group of rank $n$. The \emph{partial inner automorphism group} $I_n$ is the subgroup of $\aut{F_n}$ generated by automorphisms that conjugate one basis element by another while fixing the remaining generators.
This group has been studied in connection with basis-conjugating automorphisms and McCool-type groups; see, for instance, \cite{BN,BNS,CPVW}.

Following \cite[Section~2]{BN}, in the group $Cb_n$ we consider the elements
$$
c_{ni}=\epsilon_{1i}\epsilon_{2i}\cdots \epsilon_{ni},\,\, i=1,\ldots,n,
$$
where $\epsilon_{ii}=1$, by definition. The element $c_{ni}$ is an inner automorphism of $F_n$ which is a conjugation by an element $x_i$:
$$
c_{ni}\colon x_k\mapsto x_i^{-1}x_kx_i, \,\, k=1,\ldots, n.
$$
Define a subgroup $H_n\leq Cb_n$, $n\geq 2$, which is generated by elements $c_{ni}$, for $i=1,\ldots, n$. In a similar way, we define subgroups $H_k\leq Cb_n$, $k=2,3,\ldots,n-1$
$$
H_k=\ang{c_{k1}, c_{k2}, \ldots, c_{kk}}.
$$
The group $H_k$ is the inner automorphism group of a group $F_k=\ang{x_1, x_2, \ldots, x_k}$. We define the group $I_n$  called the \emph{partial inner automorphism group} given by
$$
I_n = \ang{H_2, H_3, \ldots, H_n}\leq Cb_n.
$$  
Follows from \cite[Theorem~1]{BN} that $I_n=\rtimes_{j=2}^n F_{j}$ is an almost-direct product of free groups. 
In particular, $I_n$ belongs to the class of groups considered in this paper.

Fix Magnus-type orderings on each free factor $F_{m_j}$ and denote by ${\mathcal P}(F_{m_j})$ the corresponding positive cones. Applying the lexicographic construction of \redefn{cone_ADP}, we obtain a positive cone $P\subset I_n$.

\begin{prop}\label{prop:In}
The subset $P \subset I_n$ defines a bi-invariant ordering on the partial inner automorphism group $I_n$.
\end{prop}

\begin{proof}
Since $I_n$ is an almost-direct product of free groups with IA-actions, the result follows directly from \rethm{main_cone}.
\end{proof}

\begin{rem}
The bi-ordering obtained above is compatible with the natural filtration of $I_n$ induced by its almost-direct product decomposition, and the corresponding subgroups are convex by \reprop{convex}.
\end{rem}

\subsection{Hypersolvable arrangements}\label{subsec:hypersolvable}

Hypersolvable arrangements form a broad class of complex hyperplane arrangements generalizing fiber-type arrangements; see \cite{FR1,FR2,JP,P}. Their complements admit rich topological and algebraic structures, and their fundamental groups have been studied extensively from both perspectives.

Let $V$ be a finite dimensional complex vector space. A finite collection of linear hyperplanes $\hyparr{A}$ contained in $V$ will be called a complex arrangement in $V$.  Let $M_{\hyparr{A}}=V\setminus \cup_{H\in \hyparr{A}}H$ denote the \emph{complement} of $\hyparr{A}$.  We define the \emph{rank} of an arrangement as $rk\, \hyparr{A} = codim_V \left( \cap_{H\in \hyparr{A}} H\right)$.

The \emph{hypersolvable arrangements} are a combinatorial generalization of the fiber-type (supersolvable) class. Now we describe the inductive steps given in \cite[Section~1]{JP} to define it. In what follows we shall consider combinatorial conditions on an \emph{arrangement pair} $(\hyparr{A}, \hyparr{B})$, $\hyparr{B}\subsetneq \hyparr{A}$. 
Set $\overline{\hyparr{B}}=\hyparr{A}\setminus \hyparr{B}$ and denote the elements of $\hyparr{B}$ by $\alpha, \beta, \gamma,\ldots$, and the elements of $\overline{\hyparr{B}}$ by $a, b, c, \ldots$. Now we give a some definitions necessary to introduce the concept of hypersolvable arrangements.

\begin{defn}
The arrangement $\hyparr{B}$ is \emph{closed} in $\hyparr{A}$ if $rk\{ \alpha, \beta, \gamma \}=3$ for any $\alpha, \beta \in \hyparr{B}$, $\alpha\neq \beta$, and any $c\in \overline{\hyparr{B}}$.
\end{defn}

\begin{defn}\label{defn:1.2}
We say that $\hyparr{B}$ is \emph{complete} in $\hyparr{A}$ if given any $a,b\in \overline{\hyparr{B}}$, $a\neq b$, there exists $\gamma \in \hyparr{B}$ such that $rk\{ a, b, \gamma \}=2$.
\end{defn}

\begin{lem}[{\cite[Lemma~1.3]{JP}}]\label{lem:1.3}

If $\hyparr{B}$ is closed and complete in $\hyparr{A}$ then
\begin{enumerate}
	\item $rk\, \hyparr{A} - rk\, \hyparr{B}\leq 1$.
	\item The element $\gamma \in \hyparr{B}$ in \redefn{1.2} is uniquely determined by $a$ and $b$. 
\end{enumerate}
\end{lem}

If $\hyparr{B}$ is closed and complete in $\hyparr{A}$ we shall put $\gamma:=f(a,b)$ in \redefn{1.2}.

\begin{defn}
The arrangement $\hyparr{B}$ is \emph{solvable} in $\hyparr{A}$ if it is closed and complete in $\hyparr{A}$ and if for any distinct elements $a, b, c \in \overline{\hyparr{B}}$ such that the elements $f(a,b)$, $f(b,c)$, $f(a,c)$ of $\hyparr{B}$ defined in \relem{1.3}\emph{(2)} are distinct one has that  $rk \{f(a,b), f(b,c), f(a,c)\}=2$.
\end{defn}

%As mentioned in \cite{JP}, the conditions described in the above definitions are independent. There is one important particular case 

We are now ready for the definition of the main object of this section.

\begin{defn}
We say that $\hyparr{A}$ is \emph{hypersolvable} if there exists a sequence of arrangements $\hyparr{A}_1\subset \cdots \subset \hyparr{A}_i\subset \hyparr{A}_{i+1}\subset \cdots \subset \hyparr{A}_{\ell}$ with $rk \hyparr{A}_1=1$, $\hyparr{A}_{\ell}=\hyparr{A}$ such that $\hyparr{A}_i$ is solvable in $\hyparr{A}_{i+1}$, for $i=1,\ldots, \ell-1$. Such a sequence will be called a \emph{composition series} of $\hyparr{A}$. 
\end{defn}

We note that, from \cite[Proposition~1.10]{JP}, if $\hyparr{A}$ is a fiber-type arrangement (supersovable) then it has a composition series as in the last definition, so the class of fiber-type arrangements is inside in the class of hypersovable ones. From \cite{FR1} we know that the fundamental group of the complement of a fiber-type arrangement is an almost-direct product of free groups, and it admits a bi-ordering as proved independently in \cite{KR} and \cite{P}.

 Let $\hyparr{A}$ be a hypersolvable arrangement. Follows from \cite[Theorem~C, item (i)]{JP} that the fundamental group of the complement of $\hyparr{A}$, $G:=\pi_1(M_{\hyparr{A}})$, is an iterated almost-direct product of free groups, the ranks of the free groups are described in \cite[Corollary~4.4]{JP}. 
Consequently, $G$ fits naturally into the framework developed in this paper.

Fix Magnus-type orderings on each free factor appearing in the almost-direct product decomposition of $G$, and let ${\mathcal P}(F_{n_j})$ denote the corresponding positive cones.
The lexicographic construction of \redefn{cone_ADP} yields a positive cone $P \subset G$.

\begin{prop}\label{prop:hypersolvable}
The group $G=\pi_1(M_{\hyparr{A}})$ admits a bi-invariant ordering defined by the positive cone $P$.
\end{prop}

\begin{proof}
The proof is an immediate consequence of \rethm{main_cone}, since $G$ is an almost-direct product of free groups.
\end{proof}

\begin{rem}
The explicit nature of the positive cone $P$ provides additional algebraic structure on the fundamental group of the complement, which may be useful in the study of homological invariants and filtrations associated with hypersolvable arrangements.
\end{rem}

\subsection{Further examples}

The same method applies to several other families of groups admitting decompositions as almost-direct products of free groups. These include, for example, certain subgroups of automorphism groups of free groups \cite{BN}, as well as groups arising from generalized configuration spaces \cite{C1,JP}. In all such cases, the lexicographic construction yields explicit positive cones and bi-invariant orderings with analogous structural properties.

\section{Remarks}\label{sec:remarks}

In this final section we collect a few remarks highlighting the scope and significance of the results obtained in this paper. Rather than pursuing further generalizations, our aim is to place the explicit construction of positive cones and bi-invariant orderings developed here in a broader algebraic and topological context.

\begin{rems}
\item[(1)]
The main purpose of this work is to make bi-invariant orderings on almost-direct products of free groups explicit through the description of their positive cones. While the existence of such orderings was previously known in many cases, the explicit nature of the cones constructed here provides a finer structural understanding of these groups.

\item[(2)]
The lexicographic construction developed in this paper is compatible with the natural filtrations arising from almost-direct product decompositions. In particular, the convexity of canonical subgroups and the compatibility with projection maps show that the resulting bi-orderings interact well with the algebraic structure of the group.

\item[(3)]
Many of the groups considered here arise as fundamental groups of configuration spaces and related topological objects. From this perspective, the explicit bi-orderings described in this work may be viewed as additional algebraic structures on these fundamental groups, potentially useful in the study of fibrations, filtrations, and group-theoretic invariants associated with such spaces.

\item[(4)]
Since every bi-orderable group admits faithful actions by order-preserving homeomorphisms of the real line, the explicit positive cones constructed here provide natural models for such actions. The compatibility properties established in \resec{properties} suggest that these actions respect the hierarchical structure induced by the almost-direct product decomposition.

\item[(5)]
The methods developed in this paper apply uniformly to a wide class of groups, including pure monomial braid groups and McCool groups. It would be natural to investigate further classes of groups admitting almost-direct product decompositions, as well as refinements of the constructions presented here in more specialized settings.

\item[(6)] From a dynamical perspective, bi-invariant orderings are closely related to faithful actions by order-preserving homeomorphisms of the real line. As discussed in \cite{DNR}, the structure of the positive cone plays a central role in understanding the qualitative features of such actions. In this sense, the explicit and computable positive cones constructed in this paper provide natural candidates for studying actions of almost-direct products of free groups on one-dimensional manifolds that reflect their internal algebraic filtrations.

\item[(7)]
Several constructions of bi-invariant orderings in geometric and topological settings are based on geometric data, such as actions on trees, foliations, or configuration spaces; see, for example, \cite{I}. Although the approach developed here is purely algebraic, it is compatible with these geometric viewpoints in the sense that the resulting orderings are defined by explicit positive cones and normal forms. This makes it possible to compare algebraic and geometric constructions of orderings within a unified framework.

\item[(8)] 
Invariant orderings under automorphism groups play an important role in the study of free groups and their automorphisms, with applications to low-dimensional topology; see \cite{RW}. The stability properties established in \resec{properties} show that the lexicographic bi-orderings constructed here are preserved by natural classes of automorphisms arising in braid groups and McCool-type groups. 
This suggests potential applications to the study of monodromy representations and topological invariants associated with configuration spaces.

\end{rems}

\appendix
\section{Reduced free groups and reduced Magnus orderings}\label{app:reduced}

This appendix recalls the notion of reduced free groups and the associated reduced Magnus orderings.
Although these objects are not required for the main results of the paper, they provide a natural extension of the constructions developed in the main text and recover results previously obtained in more general settings.

\subsection{Reduced free groups}

Let $F_n$ be the free group on generators $x_1,\dots,x_n$.
The \emph{reduced free group} $\widehat{F}_n$ is defined as the quotient of $F_n$ by the normal subgroup generated by all commutators of conjugates of generators, that is,
\[
\widehat{F}_n \;=\;
F_n \big/ \ang{ [x_i^g,x_j^h] \;\mid\; 1\leq i,j\leq n,\ g,h\in F_n }.
\]
Equivalently, $\widehat{F}_n$ is the largest quotient of $F_n$ in which conjugates of distinct generators commute. Reduced free groups arise naturally in the study of homotopy string links and related objects; see \cite{dLdM,Y}.

\subsection{Reduced Magnus expansion}

Let $\Z\langle\!\langle X_1,\dots,X_n\rangle\!\rangle$ denote the ring of formal power series in non-commuting variables. The classical Magnus expansion induces an injective homomorphism
\[
F_n \longrightarrow 1 + \langle X_1,\dots,X_n\rangle
\subset \Z\langle\!\langle X_1,\dots,X_n\rangle\!\rangle.
\]
Factoring out the ideal generated by all commutators $[X_i^g,X_j^h]$ yields a reduced Magnus expansion
\[
\widehat{F}_n \longrightarrow 1 + \langle X_1,\dots,X_n\rangle_{\mathrm{red}},
\]
which is injective; see \cite{Y}.

Ordering the target ring lexicographically induces a bi-invariant ordering on $\widehat{F}_n$, referred to as the \emph{reduced Magnus ordering}. We denote by ${\mathcal P}(\widehat{F}_n)$ the associated positive cone.

\subsection{Almost-direct products of reduced free groups}

Let
\[
\widehat{G} \;=\;
\widehat{F}_{n_k} \rtimes \widehat{F}_{n_{k-1}} \rtimes \cdots \rtimes \widehat{F}_{n_1}
\]
be an almost-direct product of reduced free groups. As in the free case, every element of $\widehat{G}$ admits a unique normal form, and the induced actions are trivial on abelianizations.

Fix reduced Magnus orderings on each factor $\widehat{F}_{n_j}$ and denote by ${\mathcal P}(\widehat{F}_{n_j})$ the corresponding positive cones. Define a subset $P \subset \widehat{G}$ lexicographically, by declaring an element positive if and only if its highest nontrivial component lies in ${\mathcal P}(\widehat{F}_{n_j})$.

\begin{prop}\label{prop:reduced}
The subset $P \subset \widehat{G}$ defines a bi-invariant ordering on the almost-direct
product $\widehat{G}$.
\end{prop}

\begin{proof}
The proof follows the same argument as in \rethm{main_cone}. The reduced Magnus ordering is bi-invariant and preserved under IA-actions, and the lexicographic construction yields a conjugation-invariant positive cone. 
\end{proof}

\begin{rem}
The results of this appendix extend the lexicographic construction of positive cones to almost-direct products of reduced free groups. We have chosen to include this material in an appendix in order to keep the main body of the paper focused on free groups, while still recording the broader scope of the method.
\end{rem}

\end{document}